\documentclass[hyp]{amsart}
\usepackage[utf8]{inputenc}
\usepackage[T1]{fontenc}
\usepackage{amsmath,amscd}
\usepackage{graphicx}
\usepackage[all]{xy}
\usepackage{amsthm}
\usepackage{amssymb,url}
\usepackage[charter]{mathdesign}
\usepackage{amsfonts}
\usepackage{mathrsfs}
\usepackage{amsxtra}
\usepackage{bbm}
\usepackage{ae}
\usepackage[all]{xy}
\usepackage{color}
\usepackage{hyperref}

\usepackage{tikz-cd}

\newtheorem{thm}{Theorem}[section]
\newtheorem{thm*}{Theorem}
\newtheorem{lem}[thm]{Lemma}
\newtheorem{cor}[thm]{Corollary}
\newtheorem{conj}[thm]{Conjecture}
\newtheorem{prop}[thm]{Proposition}

\theoremstyle{remark}
\newtheorem{remark}[thm]{Remark}

\theoremstyle{definition}
\newtheorem{deef}[thm]{Definition}

\newcommand{\R}{\mathbbm{R}}
\newcommand{\C}{\mathbbm{C}}

\newcommand{\Z}{\mathbbm{Z}}

\newcommand{\N}{\mathbbm{N}}

\newcommand{\ch}{ \mathrm{ch}}
\newcommand{\id}{ \mathrm{id}}
\newcommand{\rd}{\mathrm{d}}

\newcommand{\Bo}{\mathcal{B}}

\newcommand{\Dom}{\mathrm{Dom}\,}

\newcommand{\End}{\mathrm{End}}
\renewcommand{\epsilon}{\varepsilon}

\newcommand{\im}{\mathrm{i} \mathrm{m} \,}

\date{\today}
\title[Exotic cyclic cohomology classes]{Exotic cyclic cohomology classes and Lipschitz algebras}

\author{Magnus Goffeng, Ryszard Nest}
\address{
Magnus Goffeng,\newline
\indent Centre for Mathematical Sciences,\newline
\indent University of Lund,\newline
\indent Box 118, 221 00 LUND, Sweden\newline
\newline
\indent Ryszard Nest,\newline
\indent Department of Mathematical Sciences, \newline 
\indent University of Copenhagen, \newline
\indent Universitetsparken 5, 2100 København, Denmark \newline
}
\email{magnus.goffeng@math.lth.se, rnest@math.ku.dk}

\subjclass[2020]{58J22, 26A16, 19D55, 18F25}
\thanks{MG was supported by the Swedish Research Council Grant VR 2018-0350.}

\begin{document}
\maketitle

\begin{abstract}
We study the noncommutative geometry of algebras of Lipschitz continuous and Hölder continuous functions where non-classical and novel differential geometric invariants arise. Indeed, we introduce a new class of Hochschild and cyclic cohomology classes that pair non-trivially with higher algebraic $K$-theory yet vanish when restricted to the algebra of smooth functions. Said cohomology classes provide additional methods to extract numerical invariants from Connes-Karoubi's relative sequence in $K$-theory.
\end{abstract}

\section{Introduction}

The algebras of Lipschitz continuous and Hölder continuous functions on a compact manifold are nonseparable Banach algebras, dense in the $C^*$-algebra of continuous functions. From a classical viewpoint, the low regularity often proves unwieldely. In the framework of noncommutative geometry and cyclic cohomology, certain geometric features resembling those arising from the dense subalgebra of differentiable functions persist. The algebra of differentiable functions forms a classical object, that is also well understood in the context of noncommutative differential geometry and cyclic cohomology \cite{connesncdg,connestrace,connebook,lodaybook,teleloca} -- indeed forming a key component in the archetypical noncommutative spectral geometry coming from a Dirac operator on the underlying manifold. In this paper, we study the noncommutative differential geometry arising from the algebras of Lipschitz- and Hölder continuous functions and novel features arising therefrom. 

The space of Lipschitz and Hölder functions on a compact metric space was carefully described in the book of Weaver \cite{weaverbook}. We review some relevant parts of \cite{weaverbook} and adopt a general view on the algebra of Hölder functions as having directional derivatives along a fibre bundle (with non-separable fibres) over the sphere bundle over the manifold. This forms a classical analogue of notions of non-measurability with respect to singular traces \cite{sukolord}. We provide results that describe the cyclic (co)homology groups of the algebra of Hölder continuous functions in relation to the de Rham homology of the manifold.

The main results of this paper concern a new construction of exotic cyclic cohomology classes and Hochschild cohomology classes on the algebra of Hölder continuous functions. The cocycles are constructed from weakly summable Fredholm modules and relies on recent advances in the computation of singular traces \cite{gimpgoff2,sukolord}. The construction is related to the Connes-Karoubi sequence \cite{conneskaroubi} and Connes-Karoubi's multiplicative characters, see also the work of Kaad \cite{kaadcomparison,kaadcomp} and the work of Bunke \cite{bunkeregulator} on Borel regulators. The cohomology classes pair non-trivially with algebraic $K$-theory yet vanish on the algebra of smooth functions; indeed on any function of better than prescribed regularity.

\subsection*{Acknowledgements}
We are grateful to Heiko Gimperlein, Jens Kaad, Adam Rennie, and Alexandr Usachev for creative discussions in which most of this work took shape.

\section{Hölder and Lipschitz algebras}

Let us recall the space of Lipschitz and Hölder function of a compact metric space. At large, we follow \cite{weaverbook} and summarize its salient points. We consider a compact metric space $X$, and let $\mathrm{d}$ denote the metric. We note that for $\alpha\in (0,1)$, also $\mathrm{d}^\alpha$ is a metric. We let $\tilde{X}$ denote the space $X\cup\{\infty\}$ topologized as a disjoint union. Since $X$ is compact, $\tilde{X}$ is the one-point compactification of $X$. We metrize $\tilde{X}$ by declaring $\mathrm{d}(x,\infty):=1$ for $x\in X$. We identify $C(X)$ with an ideal in $C(\tilde{X})$ consisting of functions vanishing at the point $\infty$. For a compact metric space $(X,\mathrm{d})$ we define the Lipschitz constant of a function $f:X\to \C$ to be 
$$|f|_{\mathrm{Lip}(X,\mathrm{d})}:=\sup_{x\neq y}\frac{|f(x)-f(y)|}{\mathrm{d}(x,y)}.$$
For $\alpha\in (0,1)$, we can define the Hölder constant as $|f|_{C^\alpha(X,\mathrm{d})}:=|f|_{\mathrm{Lip}(X,\mathrm{d}^\alpha)}$. We note that for $\alpha>1$, $\mathrm{d}^\alpha$ need not be a metric but the Lipschitz constant nevertheless makes sense. The Hölder algebra is defined for $\alpha\in (0,1)$ by 
$$C^\alpha(X):=\{f:X\to \C: |f|_{C^\alpha(X,\mathrm{d})}<\infty\}.$$
We apply the convention that $C^0(X,\mathrm{d}):=C(X)$. When the metric is clear from context, we simply write $C^\alpha(X)$. To have a notation similar for Hölder and Lipschitz algebras, we write $C^{0,\alpha}(X)=\mathrm{Lip}(X,\mathrm{d}^\alpha)$ for $\alpha\in (0,1]$.

The reader should beware that for a Riemannian manifold $(M,\mathrm{d}_g)$ equipped with the geodesic distance associated with the Riemannian metric $g$, being Lipschitz continuous is a much weaker condition than being $C^1$; indeed $C^{0,1}(M)=\mathrm{Lip}(M,\mathrm{d}_g)$ is non-separable while the space of $C^1$-functions $C^1(M)$ is separable.

\begin{prop}
The space $C^\alpha(X,\mathrm{d})$ is a Banach $*$-algebra in the norm $\|f\|_{C^\alpha(X)}:=\|f\|_{C(X)}+|f|_{C^\alpha(X)}$ and if $\alpha\in (0,1]$, $C^{0,\alpha}(X,\mathrm{d})\subseteq C(X)$ is dense and closed under smooth functional calculus. 
\end{prop}

\begin{deef}
For a compact space $X$ and $\tilde{X}=X\dot{\cup}\{\infty\}$, we define 
\begin{enumerate}
\item $\Delta_X:=\{(x,x)\in\tilde{X} \times \tilde{X} : x\in \tilde{X} \}$;
\item $\Omega_X:=\tilde{X} \times \tilde{X} \setminus \Delta_X$ and
\item $\pmb{X}$ as the space of characters on $C_b(\Omega_X)/C_0(\Omega_X)$.
\end{enumerate}
The projection onto the $i$:th coordinate defines a mapping $\pi_i:\Omega_X\to \tilde{X}$, $i=1,2$. The $*$-monomorphism obtained as the composition
$$C(\tilde{X})\xrightarrow{\pi_i^*}C_b(\Omega_X)\to C(\pmb{X})$$ 
is independent of $i$ and the induced quotient mapping is denoted by $q_X:\pmb{X}\to \tilde{X}$.
\end{deef}

The notation $\pmb{X}$ is justified by it being a sort of ``thickening" of $X\cong \Delta_X$ inside $\tilde{X}\times\tilde{X}$. This is made precise for a Riemannian manifold below in Proposition \ref{thickeningprop}.

\begin{prop}
The space $\pmb{X}$ is second countable if and only if $X$ is finite. 
\end{prop}

\begin{proof}
Since $\Omega_X$ is second countable, the statement that $\pmb{X}$ is second countable is equivalent to $C_b(\Omega_X)$ being separable. The $C^*$-algebra $C_b(\Omega_X)$ is separable if and only if $\Omega_X$ is compact which holds if and only if $\Delta_X$ is open, i.e. if $X$ is finite.
\end{proof}

\begin{deef}
For $\alpha>0$, and $f\in C(\tilde{X})$, we define $\delta_\alpha(f)\in C(\Omega_X)$ by 
$$\delta_\alpha(f)(x,y):=\frac{ f(x)-f(y)}{\mathrm{d}(x,y)^\alpha}.$$
\end{deef}

We introduce the notation $\pi_L:=\pi_1^*:C(\tilde{X})\to C_b(\Omega_X)$ and $\pi_R:=\pi_2^*:C(\tilde{X})\to C_b(\Omega_X)$.

\begin{prop}
It holds that $C^\alpha(X)=\{f\in C(X): \delta_\alpha(f)\in C_b(\Omega_X)\}$. The map $\delta_\alpha:C^\alpha(X)\to C_b(\Omega_X)$ is a derivation for the bimodule structure on $C_b(\Omega_X)$ defined from $\pi_L$ and $\pi_R$, that is 
$$\delta_\alpha(f_1f_2)=\delta_\alpha(f_1)\pi_R(f_2)+\pi_L(f_1)\delta_\alpha(f_2).$$
The map $\delta_\alpha$ defines an equivalent norm on $C^\alpha(X)$, i.e. 
\begin{align*}
\|\delta_\alpha(f)\|_{C_b(\Omega_X)}&=\max(|f|_{C^\alpha(X)},\|f\|_{C(X)}),\quad f,f_1,f_2\in C^\alpha(X).
\end{align*}
In particular, $\delta_\alpha$ has closed range and $\mu\mapsto \omega_\mu(a):=\int_{\Omega_X}\delta_\alpha(a)\rd\mu$ defines an isomorphism
$$C^\alpha(X)^*\cong M(\Omega_X)/W_\alpha, \quad\mbox{where}\quad W_\alpha:=\cap_{a\in C^\alpha(X)}\left\{\mu\in M(\Omega_X):\; \int_{\Omega_X}\delta_\alpha(a)\rd\mu=0\right\}.$$
\end{prop}

\begin{proof}
It follows from a formal algebraic manipulation that $\delta_\alpha$ is a $(\pi_L,\pi_R)$-derivation if it is well defined. As such, it remains to prove that $\delta_\alpha$ is an isometry for the norm $f\mapsto \max(|f|_{C^\alpha(X)},\|f\|_{C(X)})$. For $f\in C^\alpha(X)$, 
\begin{align*}
\|\delta_\alpha(f)\|_{C_b(\Omega_X)}&=\max\left(\sup_{x\in X}\sup_{y\in\tilde{X}\setminus\{x\}} |\delta_\alpha(f)(x,y)|,\sup_{y\in X}|\delta_\alpha(f)(\infty,y)|\right)\\
&=\max\left(\sup_{x,y\in X, \, x\neq y}|\delta_\alpha(f)(x,y)|,\sup_{y\in X}|f(y)|\right)=\max(|f|_{C^\alpha(X)},\|f\|_{C(X)}).
\end{align*}
In the second step we used our conventions that $f(\infty)=0$ for $f\in C(X)$ and $\rd(y,\infty)=1$ for all $y\in X$.
\end{proof}

\begin{deef}
We define $\tilde{\delta}_\alpha:C^\alpha(X)\to C(\pmb{X})$ as the composition of $\delta_\alpha$ with the quotient mapping $C_b(\Omega_X)\to C(\pmb{X})$.
\end{deef}

\begin{prop}
The mapping $\tilde{\delta}_\alpha$ is a derivation for the $q_X^*$-action of $C^\alpha(X)$ on $C(\pmb{X})$. 
\end{prop}

\begin{proof}
It is clear that $\tilde{\delta}_\alpha$ is a derivation for the $q_X^*$-action from the fact that $\delta_\alpha$ is a $(\pi_L,\pi_R)$-derivation. 
\end{proof}

A metric space $(X,\rd)$ is said to separate points uniformly \cite[Definition 4.10]{weaverbook} if there is a constant $a>1$ such that for any $x,y\in X$ there is an $f\in \ker(\tilde{\delta}_1)$ with $\|\delta_1f\|_{C_b(\Omega_X)}\leq a$ and $|f(x)-f(y)|=\rd(x,y)$. 

\begin{prop}
\label{blknaolknad}
The small Hölder algebra $h^\alpha(X):=\ker(\tilde{\delta}_\alpha)$ satisfies 
\begin{itemize}
\item $h^\alpha(X)^{**}=C^\alpha(X)$ as soon as $(X,\rd^\alpha)$ separates points uniformly. 
\item If $0<\beta<\alpha<1$ then 
$$h^\alpha(X)=\overline{C^\beta(X)}^{C^\alpha(X)}.$$
\end{itemize} 
In particular, if $X$ is a compact Riemannian manifold with its geodesic distance $\rd$ and $\alpha\in (0,1)$, then $h^\alpha(X)^{**}=C^\alpha(X)$ and 
$$h^\alpha(X)=\overline{C^\infty(X)}^{C^\alpha(X)}.$$

\end{prop}

\begin{proof}
The first statement can be found in \cite[Theorem 4.38]{weaverbook}. The second statement follows from \cite[Theorem 2.52 and 8.27]{weaverbook}. Finally, the last conclusion follows from the first two statements and a standard mollification argument.
\end{proof}

\begin{prop}
Consider the flip mapping $\sigma:\Omega_X\to \Omega_X$, $(x,y)\mapsto (y,x)$. The flip mapping induces a non-trivial involution $\tilde{\sigma}$ on $\pmb{X}$. Elements of $\delta_\alpha(C^\alpha(X))\subseteq C_b(\Omega_X)$ and $\tilde{\delta}_\alpha(C^\alpha(X))\subseteq C(\pmb{X})$ are odd under $\sigma^*$ and $\tilde{\sigma}^*$, respectively.

\end{prop}

\subsection{The Hölder algebra on a manifold}

\begin{prop}
\label{thickeningprop}
Let $X$ be a Riemannian manifold equipped with its geodesic metric. Then there is a unique quotient mapping $q_S:\pmb{X}\to SX$ factoring $q_X:\pmb{X}\to X$ over the projection mapping $\pi_S:SX\to X$ and making the following diagram commutative
\begin{equation}
\label{comdiaone}
\begin{tikzcd}
C^1(X)\arrow[r,"\mathrm{d}"]\arrow[d]&C(SX)\arrow[d,"q_S^*"]   \\
\mathrm{Lip}(X)\arrow[r, "\tilde{\delta}_1"]  &C(\pmb{X}) ,
\end{tikzcd}
\end{equation}
where the left vertical arrow is the inclusion and the top horizontal mapping is defined from the exterior derivative $\rd:C^1(X)\to C(X,T^*X)$ when identifying sections in $C(X,T^*X)$ with functions on $SX$ via the pairing of tangent vectors with cotangent vectors. Moreover, 
$$C^1(X)=\{f\in C(X): \tilde{\delta}_1(f)\in q_S^*C(SX)\}.$$
\end{prop}

As a consequence of Proposition \ref{thickeningprop}, we note that for $v\in SX$, and $\omega\in q_S^{-1}(v)\subseteq \pmb{X}$, the functional $\mathrm{Lip}(X)\ni f\mapsto [\tilde{\delta}_1(f)](\omega)$ extends the functional $C^1(X)\ni f\mapsto v.\mathrm{d}f(x)$, where $x=\pi_S(v)$, and we are using the inclusion $SX\to TX$ and the pairing of $TX$ with $T^*X$.

\begin{proof}
Since $\mathrm{d}C^1(X)\subseteq C(SX)$ generates $C(SX)$ as a $C^*$-algebra, the diagram \eqref{comdiaone} determines $q_S$ from $\mathrm{d}C^1(X)$ whenever $q_S$ factors over $q_X$. Thus it suffices to construct $q_S$. Pick a tubular neighborhood $N\subseteq \tilde{X}\times \tilde{X}$ of $\Delta_X$ and a diffeomorphism $\nu:N\cong N^*X$ with a normal bundle $N^*X\to X$. Let $R:N^*X\setminus X\to SX$ denote the projection defined from the quotient of the $\R_+$-action. After picking a $\chi\in C_c(N^*X,[0,1])$ with $\chi=1$ near the zero section $X\subseteq N^*X$, we can define a linear mapping $\tau:C(SX)\to C_b(\Omega_X)$ by $C^1(SX)\xrightarrow{R^*} C^1(N^*X\setminus X)\cap C_b(N^*X\setminus X)C_0(N^*X)$ with the inclusion $C_b(N^*X\setminus X)C_0(N^*X)\xrightarrow{\nu^*} C_b(\Omega_X)$. The mapping $\tau$ is readily verified to be multiplicative modulo $C_0(\Omega_X)$ and as such $q_S$ is well defined as the dual mapping to the composition $C(SX)\xrightarrow{\tau}C_b(\Omega_X)\to C(\pmb{X})$. 

The last identity in the proposition is easily proved from that whenever $\tilde{\delta}_1(f)\in q_S^*C(SX)$ then all partial derivatives of $f$ are continuous and $f$ is by definition $C^1$. 
\end{proof}

We consider the special case of the Riemannian manifold $X=S^1$. We identify $S^1$ with the unit circle in $\C$ parametrized by a complex parameter $z$, $|z|=1$, or sometimes denoted by $w$. We can refine the derivation $\delta_1$ by considering 
$$\delta_{S^1}(f)(z,w):=\frac{f(z)-f(w)}{z-w}.$$ 
In a similar way, one defines $\tilde{\delta}_{S^1}:\mathrm{Lip}(S^1)\to C(\pmb{S}^1)$. Then $f\in C^1(S^1)$ if and only if $\tilde{\delta}_{S^1}(f)\in q_{S^1}^*C(S^1)$ and $\delta_{S^1}(f)=q_{S^1}^*(f')$. In this particular instance, the non-separable algebra $L^\infty(S^1)$ provides a more natural object to us when studying $\mathrm{Lip}(S^1)$ rather than $C(\pmb{S}^1)$. 

\begin{prop}
\label{infeindaodn}
Let $X$ be a compact Riemannian manifold and $x_0\in X$ a point. Define $C^\alpha_0(X\setminus \{x_0\}):=C^\alpha(X)\cap C_0(X\setminus \{x_0\})$ and define $\rd_{x_0}^\alpha\in  C^\alpha_0(X\setminus \{x_0\})$ by $\rd_{x_0}^\alpha(x):=\rd(x_0,x)^\alpha$. Then both of the spaces 
\begin{itemize}
\item $C^\alpha(X)/\overline{\rd_{x_0}^\alpha C^\alpha_0(X\setminus \{x_0\})}$
\item $C^\alpha_0(X\setminus \{x_0\})/\overline{(C^\alpha_0(X\setminus \{x_0\}))^2}$
\end{itemize}
are infinite-dimensional.
\end{prop}

This proposition shows that any infinitesimal approach to Hölder geometries require infinite dimensional model spaces.

\begin{proof}
Pick a small enough $\epsilon>0$ so that $B_{2\epsilon}(x_0)$ is diffeomorphic to a euclidean ball $B_2(0)\subseteq \R^n$. We pick a diffeomorphism $B_2(0)\setminus \overline{B_1(0)}\cong B_2(0)\setminus \{0\}$ that extends to a uniformly Lipschitz mapping on $B_2(0)$. We extend this diffeomorphism to a diffeomorphism $X\setminus \overline{B_\epsilon(x_0)}\cong X\setminus \{x_0\}$ that extends by continuity to a Lipschitz mapping on $X$. We denote the image of $C^\infty(X\setminus B_\epsilon(x_0))$ under pullback along this diffeomorphism by $V_0\subseteq C_b(X\setminus \{x_0\})$. Define $V_{x_0}:=\rd^\alpha_{x_0}V_0\subseteq C^\alpha_0(X\setminus \{x_0\})$. The proposition follows if we can show that $V_{x_0}/(V_{x_0}\cap \overline{(C^\alpha_0(X\setminus \{x_0\}))^2})$ is infinite-dimensional. We restrict our attention to the case $n>1$ where the existence of an injection $C^\infty(S^{n-1})/\C1\to V_{x_0}/(V_{x_0}\cap \overline{(C^\alpha_0(X\setminus \{x_0\}))^2})$ proves the result. The case $n=1$ goes similarly.
\end{proof}

\begin{remark}
Proposition \ref{infeindaodn} provides a reason for why one can expect the cyclic cohomology of Hölder spaces to be complicated: the ordinary methods of localization applying in presence of higher regularity (cf. \cite[Chapter 3]{teleloca}) are not applicable. 
\end{remark}

\subsection{Operator algebra structure}

The Hölder algebra admits additional structure. We equip it with the structure of an operator $*$-algebra. Recall that an operator algebra is a Banach algebra which is matrix normed and a completely bounded multiplication. Equivalently, an operator algebra is isomorphic to a closed subalgebra of a $C^*$-algebra. For more details on operator algebras, see \cite{blechermerdy}. The notion of an operator $*$-algebra was developed in \cite{blekames}. The following two propositions are clear from construction.

\begin{prop}
\label{adknaodkna}
Let $X$ be a compact metric space. The mapping 
$$i_\alpha:C^\alpha(X)\to C_b(\Omega_X,M_2(\C)), \quad a\mapsto \begin{pmatrix} \pi_L(a)&0\\ \delta_\alpha(a)& \pi_R(a)\end{pmatrix},$$
realizes $C^\alpha(X)$ as a closed subalgebra of a $C^*$-algebra and induces an operator $*$-algebra structure on $C^\alpha(X)$.
\end{prop}

\begin{prop}
Let $M$ be a compact Riemannian manifold and fix a first order differential operator $D$ acting between two hermitean vector bundles $E,F\to M$ such that $\sigma_2(D^*D)(x,\xi)=|\xi|^2\id_E$ coincides with the metric on $M$. The mapping 
$$i_D:\mathrm{Lip}(M)\to L^\infty(M, \End(E\oplus F)), \quad a\mapsto \begin{pmatrix} a&0\\ [D,a]& a\end{pmatrix}.$$
realizes $\mathrm{Lip}(M)$ as a weak *-closed subalgebra of $L^\infty(M, \End(E\oplus F))$.
\end{prop}

\section{Cyclic (co)homology}

We shall be interested in the cyclic cohomology of Hölder algebras. In this section we recall the construction of cyclic (co)homology of a unital Banach algebra. We write $\otimes^{\rm alg}$ for the algebraic tensor product and $\tilde{\otimes}$ for the projective tensor product. The constructions' dependence on the Banach structure is through the tensor product, and applies to any choice of tensor product making products and flip maps continuous. 

For a Banach algebra $\mathfrak{A}$ we write 
$$C_k(\mathfrak{A}):=\mathfrak{A}^{\tilde{\otimes} k+1} \quad\mbox{and}\quad C^k(\mathfrak{A}):=C_k(\mathfrak{A})^*.$$
In particular, $C^k(\mathfrak{A})$ consists of $k+1$-linear functionals $c$ on $\mathfrak{A}$ such that $|c(a_0,\ldots, a_k)\|\lesssim \|a_0\|\cdots \|a_k\|$. We consider the cyclicity operator 
$$\Lambda:C_k(\mathfrak{A})\to C_k(\mathfrak{A}),\quad \Lambda(a_0\otimes a_1\otimes \cdots\otimes a_k):=(-1)^{k}a_1\otimes \cdots\otimes a_k\otimes a_0,$$
and the Hochschild boundary operator 
\begin{align*}
b&:C_k(\mathfrak{A})\to C_{k-1}(\mathfrak{A}),\\
b&(a_0\otimes a_1\otimes \cdots\otimes a_k):=\sum_{j=0}^{k-1}(-1)^{j}a_0\otimes \cdots\otimes a_ja_{j+1}\otimes\cdots \otimes \cdots \otimes a_k+\\
&\qquad\qquad\qquad\qquad\qquad +(-1)^k a_ka_0\otimes a_1\otimes \cdots\otimes a_{k-1}.
\end{align*}
We write $C_k^\lambda(\mathfrak{A}):=C_k(\mathfrak{A})/(1-\Lambda)C_k(\mathfrak{A})$ and define $C^k_\lambda(\mathfrak{A})\subseteq C^k(\mathfrak{A})$ as the sub-complex of $\Lambda$-invariant elements. 

\begin{deef}
Let $\mathfrak{A}$ be a unital Banach algebra. 

The Hochschild homology of $\mathfrak{A}$ is defined as the homology of $(C_*(\mathfrak{A}),b)$ and is denoted by $HH_*(\mathfrak{A})$. The Hochschild cohomology of $\mathfrak{A}$ is defined as the cohomology of $(C^*(\mathfrak{A}),b)$ and is denoted by $HH^*(\mathfrak{A})$.

The cyclic homology of $\mathfrak{A}$ is defined as the homology of $(C_*^\lambda(\mathfrak{A}),b)$ and is denoted by $HC_*^\lambda(\mathfrak{A})$. The cyclic cohomology of $\mathfrak{A}$ is defined as the cohomology of $(C^*_\lambda(\mathfrak{A}),b)$ and is denoted by $HC^*_\lambda(\mathfrak{A})$.
\end{deef}

We use the notation $HC_*^\lambda(\mathfrak{A})$ to distinguish the model we use from that of the bi-complex normally used to define $HC_*(\mathfrak{A})$, even if the two produce isomorphic results. We sometimes consider reduced (co)-homology groups, i.e. $\ker(b)/\overline{\im(b)}$, which we indicate by an additional (super)subscript red. For instance, $HC_*^{\lambda, \rm red}(\mathfrak{A}):=\ker(b:C_k^\lambda(\mathfrak{A})\to C_{k-1}^\lambda(\mathfrak{A}))/\overline{bC_{k+1}^\lambda(\mathfrak{A})}$.

\subsection{Connes' SBI-sequence, Chern character and character formula} 

Cyclic (co)homology comes with a rich structural theory, see for instance \cite{connesncdg,connebook,lodaybook}. The following result of Connes relates cyclic (co)homology with Hochschild (co)homology, a result proven using the cyclic bicomplex.

\begin{thm}[Connes \cite{connebook}, Chapter III.$1.\gamma$]
\label{sbithm}
Let $\mathfrak{A}$ be a Banach algebra. Then there is a long exact sequence in homology 
$$\cdots \xrightarrow{I} HC_{*+2}^\lambda(\mathfrak{A})\xrightarrow{S}HC_{*}^\lambda(\mathfrak{A})\xrightarrow{B}HH_{*+1}(\mathfrak{A})\xrightarrow{I} HC_{*+1}^\lambda(\mathfrak{A})\xrightarrow{S}HC_{*-1}^\lambda(\mathfrak{A})\xrightarrow{B}\cdots, $$
and a long exact sequence in cohomology.
$$\cdots \xrightarrow{B} HC^{*-2}_\lambda(\mathfrak{A})\xrightarrow{S}HC^{*}_\lambda(\mathfrak{A})\xrightarrow{I}HH^{*}(\mathfrak{A})\xrightarrow{B} HC_\lambda^{*-1}(\mathfrak{A})\xrightarrow{S}HC^{*+1}_\lambda(\mathfrak{A})\xrightarrow{I}\cdots, $$
\end{thm}

The sequence from Theorem \ref{sbithm} is called the SBI-sequence.

\begin{deef}
A bounded Fredholm module on $\mathfrak{A}$ is a triple $(\pi, \mathcal{H},F)$ where $\pi$ is a bounded representation of $\mathfrak{A}$ on the Hilbert space $\mathcal{H}$ and $F$ is a bounded operator such that 
\begin{equation}
\label{kknoknad}
[F,b],\ b(F^*-F), \ b(F^2-1)\in \mathbb{K}(\mathcal{H}), \quad \forall b\in \pi(\mathfrak{A})\cup \pi(\mathfrak{A})^*.
\end{equation}
If $b(F^*-F)= b(F^2-1)=0$, we say that $(\pi, \mathcal{H},F)$ is a strict Fredholm module. 

When $\mathcal{H}$ is equipped with a $\Z/2$-grading, in which $F$ is odd and $\mathfrak{A}$ acts as even operators,  we say that $(\pi, \mathcal{H},F)$ is an even Fredholm module, or of parity $0$. Otherwise we say that $(\pi, \mathcal{H},F)$  is an odd Fredholm module, or of parity $1$. 
\end{deef}

We recall the space of Schatten class operators 
$$\mathcal{L}^p(\mathcal{H})=\{T\in \mathbb{K}(\mathcal{H}): \mathrm{Tr}(|T|^p)<\infty\},$$
and the space of weak Schatten class operators 
$$\mathcal{L}^{(p,\infty)}(\mathcal{H})=\{T\in \mathbb{K}(\mathcal{H}): \mu_k(T)=O(k^{-1/p})\}.$$
Here $\mu_k$ denotes the $k$:th singular value. See more in \cite{simonbook} or \cite[Chapter IV, Appendix C]{connebook}. A Fredholm module $(\pi, \mathcal{H},F)$ of $\mathfrak{A}$ is said to be $p$-summable if 
$$
[F,b]\in \mathcal{L}^p(\mathcal{H}),\ b(F^*-F), \ b(F^2-1)\in \mathcal{L}^{p/2}(\mathcal{H}), \quad \forall b\in \pi(\mathfrak{A})\cup \pi(\mathfrak{A})^*.
$$
Similarly, one defines the notion of a $(p,\infty)$-summable Fredholm module. Examples of $(p,\infty)$-summable Fredholm modules can be found below in Section \ref{examoedonaodhol}.

We write $\mathrm{STr}:\mathcal{L}^1(\mathcal{H})\to \C$ for the supertrace when $\mathcal{H}$ is $\Z/2$-graded. For notational convenience, we use the same notation also for the trace when $\mathcal{H}$ is ungraded. For an extended limit $\omega$, i.e. a state $\omega\in \ell^\infty(\N)^*$ annihilating $c_0(\N)$, we let $\mathrm{STr}_\omega:\mathcal{L}^1(\mathcal{H})\to \C$ denote the associated super-Dixmier trace with the same convention as above if $\mathcal{H}$ is ungraded. Write 
$$\#n:=\begin{cases} 
0, \;&\mbox{$n$ even},\\
1, \;&\mbox{$n$ odd}.\end{cases}$$

\begin{thm}
\label{conoancoand}
Let  $(\pi, \mathcal{H},F)$ be a strict $n+1$-summable bounded Fredholm module of parity $\#n$. Then the cochain 
$$\mathrm{ch}_{CC}^n(F).(a_0,\ldots, a_n):=c_n\mathrm{STr}\left(F[F,a_0]\cdots [F,a_{n}]\right),$$
is a cyclic $n$-cocycle. Here 
$$c_n=
\begin{cases}
(-1)^{n(n-1)/2}\Gamma(n/2+1),\; &\mbox{for $n$ even},\\
\sqrt{2i}(-1)^{n(n-1)/2}\Gamma(n/2+1),\; &\mbox{for $n$ odd},
\end{cases}$$ 
is choosen so that the class $[\mathrm{ch}_{CC}^n(F)]\in HC^n_\lambda(\mathfrak{A})$ satisfies that 
$$S[\mathrm{ch}_{CC}^n(F)]=[\mathrm{ch}_{CC}^{n+2}(F)].$$
\end{thm}

The reader can find a proof of Theorem \ref{conoancoand} in \cite[Chapter IV.$1.\beta$]{connebook}. Let $K_*^{\rm top}(\mathfrak{A})$ denote the topological $K$-theory of $\mathfrak{A}$. The class $[\mathrm{ch}_{CC}^n(F)]\in HC^n_\lambda(\mathfrak{A})$ is called the Connes-Chern character of $(\pi, \mathcal{H},F)$ and is of interest as it in the pairing 
$$K_*^{\rm top}(\mathfrak{A})\times HC^*_{\lambda}(\mathfrak{A})\to \C,$$
computes the pairing with the $K$-homology class $[F]\in K^*(\mathfrak{A})$ under the index pairing 
$$\langle\cdot,\cdot\rangle: K_*^{\rm top}(\mathfrak{A})\times K^*(\mathfrak{A})\to \Z.$$

\begin{deef}
An unbounded Fredholm module on $\mathfrak{A}$ is a triple $(\pi, \mathcal{H},D)$ where $\pi$ is a bounded representation of $\mathfrak{A}$ on the Hilbert space $\mathcal{H}$ and $D$ is a densely defined self-adjoint operator on $\mathcal{H}$ such that 
\begin{enumerate}
\item For any $b\in \pi(\mathfrak{A})\cup \pi(\mathfrak{A})^*$, we have that $b\Dom(D)\subseteq \Dom(D)$ and $[D,b]$ is bounded in the norm of $\mathcal{H}$. 
\item For any $b\in \pi(\mathfrak{A})\cup \pi(\mathfrak{A})^*$, the operator $b(i\pm D)^{-1}$ is compact on $\mathcal{H}$. 
\end{enumerate}

When $\mathcal{H}$ is equipped with a $\Z/2$-grading, in which $D$ is odd and $\mathfrak{A}$ acts as even operators,  we say that $(\pi, \mathcal{H},D)$ is an even Fredholm module, or of parity $0$. Otherwise we say that $(\pi, \mathcal{H},D)$  is an odd Fredholm module, or of parity $1$.

If for any $b$, we have $b(i\pm D)^{-1}\in \mathcal{L}^p(\mathcal{H})$ or $b(i\pm D)^{-1}\in \mathcal{L}^{(p,\infty)}(\mathcal{H})$ we say that $(\pi, \mathcal{H},D)$ is $p$-summable or $(p,\infty)$-summable, respectively.
\end{deef}

It follows from the results of \cite{sww} that if $(\pi, \mathcal{H},D)$ is an unbounded Fredholm module, then $(\pi, \mathcal{H},D|D|^{-1})$ is a bounded Fredholm module. In fact, the operator estimate of \cite{sww} implies that if $(\pi, \mathcal{H},D)$ is $p$-summable or $(p,\infty)$-summable, then so is $(\pi, \mathcal{H},D|D|^{-1})$.

\begin{thm}
\label{connescharacter}
Let  $(\pi, \mathcal{H},D)$ be an $(n,\infty)$-summable unbounded Fredholm module of parity $\#n$. Fix an extended limit $\omega$. Then the cochain 
$$\mathrm{Hoch}^n_\omega(D).(a_0,\ldots, a_n):=\mathrm{STr}_\omega\left(a_0[D,a_1]\cdots [D,a_{n}](1+D^2)^{-n/2}\right)$$
is a Hochschild $n$-cocycle. The class $[\mathrm{Hoch}^n_\omega(D)]\in HH^n(\mathfrak{A})$ is independent of $\omega$ and satisfies that 
$$[\mathrm{Hoch}^n_\omega(D)]=I[\mathrm{ch}_{CC}^{n}(D|D|^{-1})].$$
\end{thm}

Theorem \ref{connescharacter} is called Connes' character formula and its proof can be found in \cite[Chapter IV.$2.\gamma$]{connebook}.

\begin{cor}
\label{onoandaoin}
Let  $(\pi, \mathcal{H},D)$ be an $(n,\infty)$-summable unbounded Fredholm module of parity $\#n$. Assume that
$$[\mathrm{Hoch}^n(D)]\neq 0\in HH^{n}(\mathfrak{A}).$$
Then $F:=D|D|^{-1}$ does not make $(\pi, \mathcal{H},F)$ into an $n-1$-summable bounded Fredholm module of parity $\#n$.
\end{cor}

\begin{proof}
If $(\pi, \mathcal{H},F)$ is an $n-1$-summable bounded Fredholm module of parity $\#n$, then by Theorem \ref{sbithm} and \ref{connescharacter}, 
$$[\mathrm{Hoch}^n(D)]=I[\mathrm{ch}_{CC}^{n}(D|D|^{-1})]=IS[\mathrm{ch}_{CC}^{n-2}(D|D|^{-1})]=0.$$
\end{proof}

\subsection{The Connes-Karoubi sequence and pairings with cohomology}

For Banach algebras, the pairing of Connes' SBI sequence with algebraic $K$-theory, and a relative group for the map into topological $K$-theory, was studied by Connes-Karoubi \cite{conneskaroubi}. We recall it here and provide an application to obstructions against finite summability of Fredholm modules therenext.

\begin{thm}[Connes-Karoubi \cite{conneskaroubi}]
\label{ckadoandona}
Let $\mathfrak{A}$ be a Banach algebra. Then there is a commuting diagram with exact rows:
\[
\begin{CD}
\ldots@>>>K_{n+1}^{\rm top}(\mathfrak{A}) @>>> K_n^{\rm rel}(\mathfrak{A})  @>>> K_n^{\rm alg}(\mathfrak{A})@>>>K_n^{\rm top}(\mathfrak{A})@>>>\ldots  \\
@.@V\ch_{n+1} VV@V\ch_n^{\rm rel}VV @V\frac{1}{n!} D_nVV@V\ch_n VV @.  \\
\ldots@>>>HC^\lambda_{n+1}(\mathfrak{A}) @>S>> HC^\lambda_{n-1}(\mathfrak{A})@>\frac{1}{n}B>>HH_n(\mathfrak{A})@>I>>HC^\lambda_n(\mathfrak{A})@>>>\ldots  \\
\end{CD}.\]
\end{thm}

Here $K_*^{\rm alg}(\mathfrak{A})$ denotes the algebraic $K$-theory of $\mathfrak{A}$ and $K_*^{\rm rel}(\mathfrak{A})$ is the relative group for the forgetful map $K_*^{\rm alg}(\mathfrak{A})\to K_*^{\rm top}(\mathfrak{A})$ considered in \cite{conneskaroubi}. In the same spirit as in Corollary \ref{onoandaoin}, Theorem \ref{ckadoandona} provides an obstruction in algebraic $K$-theory to finite summability.

\begin{cor}
\label{onoandaoinvariation}
Let $n\in \N$, and $(\pi, \mathcal{H},F)$ be a bounded Fredholm module of parity $\#n$ for a Banach algebra $\mathfrak{A}$. Assume that there is an element 
$$x\in \mathrm{im}(K_n^{\rm alg}(\mathfrak{A})\to K_n^{\rm top}(\mathfrak{A})),$$
such that the index pairing is nontrivial:
$$\langle x,[(\pi, \mathcal{H},F))\rangle \neq 0.$$
Then $F$ does not define an $n-1$-summable bounded Fredholm module of parity $\#n$.
\end{cor}

\begin{proof}
We can assume that  $(\pi, \mathcal{H},F)$ is $n+1$-summable because if not the statement of the corollary holds trivially. Let $y\in K_n^{\rm alg}(\mathfrak{A})$. By Connes-Karoubi's theorem we have that 
$$\langle j_*y,[(\pi, \mathcal{H},F))\rangle =\ch_{CC}^n(F).\ch(j_*y)=\frac{1}{n!} \ch_{CC}^n(F). \left[ID_n(y)\right]=\frac{1}{n!} \left[I\ch_{CC}^n(F)\right].D_n(y).$$
If $F$ is $n-1$-summable bounded Fredholm module of parity $\#n$, then $I\ch_{CC}^n(F)=IS\ch_{CC}^{n-2}(F)=0$. Therefore the assumption of the corollary contradicts $n-1$-summability.
\end{proof}

\section{Properties of maps induced from $C^\alpha\hookrightarrow C^\beta$}

We are interested in the cyclic (co)homology groups of the algebra of Hölder continuous functions. In large degrees, much can be said using the inclusion $C^\infty\hookrightarrow C^\alpha$ and the Connes-Hochschild-Kostant-Rosenberg theorem. We study large degrees in Subsection \ref{lnadpakdnaknd}. In low degrees, less is known and we give results showing that $C^\infty\hookrightarrow C^\alpha$ provides little information. We study low degrees in Subsection \ref{knlkansdkoanda}.

\subsection{The map $C^\infty(X)\to C^\alpha(X)$ in large degrees}
\label{lnadpakdnaknd}

The cyclic (co)homology groups of $C^\infty(X)$ were computed by Connes \cite[Chapter III.$2.\alpha$]{connebook} using the Hochschild-Kostant-Rosenberg theorem and the SBI-sequence.

\begin{thm}[Connes-Hochschild-Kostant-Rosenberg theorem]
Let $X$ be a compact manifold. Let $\Omega^k(X)$ denote the space of smooth $k$-forms on $X$ with the exterior differential $\rd$. Let $\Omega_k(X)$ denote the space of $k$-currents and $Z_k(X)\subseteq \Omega_k(X)$ the space of closed $k$-currents. Then there are natural isomorphisms
$$HC_k^\lambda(C^\infty(X))\cong 
\begin{matrix} 
\Omega^k(X)/\rd \Omega^{k-1}(X)\\
\oplus\\
\oplus_{j=1}^\infty H^{k-2j}_{\rm dR}(X) 
\end{matrix}\quad
\mbox{and}
\quad HC^k_\lambda(C^\infty(X))\cong 
\begin{matrix} 
Z_k(X)\\
\oplus\\
\oplus_{j=1}^\infty H_{k+2j}^{\rm dR}(X) 
\end{matrix}$$
\end{thm}

The SBI-sequence in the case of $C^\infty(X)$ can be described further; in terms of negative cyclic homology this was done in \cite[Chapter 5.1.12]{lodaybook}. Using the Connes-Hochschild-Kostant-Rosenberg theorem we can deduce a relation between classical de Rham (co)homology groups and large degree cyclic (co)homology groups of the algebra of Hölder continuous functions.

\begin{thm}
Let $X$ be a compact manifold, $\alpha\in (0,1]$, and $k>n/\alpha$. 
\begin{itemize}
\item The induced map 
$$HC_k^\lambda(C^\infty(X))\to HC_k^\lambda(C^\alpha(X)),$$
is injective.
\item The induced map 
$$HC^k_\lambda(C^\alpha(X))\to HC^k_\lambda(C^\infty(X)),$$
is surjective.
\end{itemize}
\end{thm}

\begin{proof}
By the Connes-Hochschild-Kostant-Rosenberg theorem, the pairing $HC_k^\lambda(C^\alpha(X))\times HC_\lambda^k(C^\alpha(X))\to \C$ is perfect for $k>n$ so the injectivity of $HC^\lambda_k(C^\infty(X))\to HC^\lambda_k(C^\alpha(X))$ follows from the surjectivity of $HC_\lambda^k(C^\alpha(X))\to HC_\lambda^k(C^\infty(X))$. Since the Chern character induces an isomorphism $\ch_*:K_*(X)\otimes \R\to H_*^{\rm dR}(X)$ and Poincaré duality $K^*(T^*X)\to K_*(X)$ is an isomorphism whose range at the level of cycles consists of bounded Fredholm modules defined from zeroth order pseudo-differential operators, the theorem now follows from \cite[Theorem 3.5]{goffeven}, see also \cite{goffodd}.
\end{proof}

\subsection{Approximating the diagonal and the inclusion $C^\beta\hookrightarrow C^\alpha$ if $\beta>2\alpha$}
\label{knlkansdkoanda}

\begin{thm}
\label{approxomtienri}
Let $X$ be a compact metric space. Then the space of Hochschild cocycles satisfies
$$Z_1(C^\beta(X))\subseteq \overline{B_1(C^\alpha(X))}^{C^\alpha(X^{2})},$$
for $\beta>\alpha$. If $X$ is a compact Riemannian manifold and $\beta>2\alpha$, the same is true with $C^\alpha(X)^{\otimes_{\rm min} 2}$ replacing $C^\alpha(X^{2})$.
\end{thm}

To prove this theorem, we need a lemma. We note that for the minimal tensor product, the condition $\beta>2\alpha$ can at best be relaxed to $\beta\geq 2\alpha$ in general as per the remark after Proposition \ref{lkandapdn}.  

\begin{lem}
Let $X$ be a compact metric space. Then there is a sequence $(\Delta_j)_{j\in \N}\subseteq \mathrm{Lip}(X)$ approximating the diagonal in the sense that 
\begin{enumerate}
\item[i)] $\Delta_j(x,x)=1$ for $x\in X$;
\item[ii)] $\|\Delta_j\|_{\mathrm{Lip}(X\times X)}=O(j)$;
\item[iii)] $\mathrm{supp}(\Delta_j)\subseteq \{(x,y)\in X\times X: \rd(x,y)<\frac{1}{j}\}$.
\end{enumerate}
If $X$ is a Riemannian manifold, we can take $(\Delta_j)_{j\in \N}\subseteq \mathrm{Lip}(X)$  such that additionally 
\begin{enumerate}
\item[iv)] $\|(\delta_1\otimes \delta_1)(\Delta_j)\|_{C_b(\Omega_X\times \Omega_X)}=O(j^2)$.
\end{enumerate}
\end{lem}

\begin{proof}
Take a positive, decreasing function $\chi\in C^\infty_c[0,1)$ with $\chi(t)=1$ in a neighbourhood of $t=0$. It is a short computation to verify that 
$$\Delta_j(x,y):=\chi(j\rd(x,y)),$$
satisfies the stated properties.
\end{proof}

\begin{proof}[Proof of Theorem \ref{approxomtienri}]
Let $s$ denote the map $C_k(\mathfrak{A})\to C_{k+1}(\mathfrak{A})$, $s(x):=1\otimes x$. Define the linear maps 
$$T_j:C_1(C^\beta(X))\to C^\alpha(X^{3}), \quad T_j(x):=(\Delta_j\otimes 1)\cdot s(x), \quad j\in \N.$$
For $x=\sum_l f_{l,0}\otimes f_{l,1}$, we compute that 
$$bT_j(x)=x+\sum_l \Delta_j\cdot (f_{l,0}f_{l,1}\otimes 1-1\otimes f_{l,0}f_{l,1}).$$
As such, the first part of the theorem is proved upon showing that $\Delta_j(1\otimes f-f\otimes 1)\to 0$ in $C^\alpha(X\times X)$ uniformly in $f\in C^\beta(X)$ if $\beta>\alpha$.  The second statement in the minimal tensor product is proven upon showing that $\Delta_j(1\otimes f-f\otimes 1)\to 0$ in $C^\alpha(X)\otimes_{\rm min} C^\alpha(X)$ uniformly in $f\in C^\beta(X)$ if $\beta>2\alpha$. The second statement is proven along similar lines as the first, using the injectivity of the minimal tensor product and the embedding $i_\alpha$ from Proposition \ref{adknaodkna}, and is therefore omitted.

Write $U_j=\{(x,y)\in X\times X: \rd(x,y)<\frac{1}{j}\}$. We have that $\|\cdot\|_{C^\alpha}=\|\cdot\|_{C^0}+|\cdot|_{C^\alpha}$. We estimate that 
$$\|\Delta_j(1\otimes f-f\otimes 1)\|_{C^0(X\times X)}=\|\Delta_j(1\otimes f-f\otimes 1)\|_{C^0(U_j)}\leq \|1\otimes f-f\otimes 1\|_{C^0(U_j)}\lesssim |f|_{C^\beta}j^{-\beta}.$$
Note that 
\begin{align*}
|1\otimes f-f\otimes 1|_{C^\alpha(U_j)}\leq &|1\otimes f-f\otimes 1|_{C^0(U_j)}^{1-\frac{\alpha}{\beta}}|1\otimes f-f\otimes 1|_{C^\beta(U_j)}\leq j^{\frac{\alpha}{\beta}-1}\|f\|_{C^\beta(X)},
\end{align*}
and by a similar trick 
\begin{align*}
|\Delta_j|_{C^\alpha(U_j)}\lesssim& j^{\alpha}.
\end{align*}
We can therefore estimate that 
\begin{align*}
|\Delta_j(1\otimes f-f\otimes 1)&|_{C^\alpha(X\times X)}=|\Delta_j(1\otimes f-f\otimes 1)|_{C^\alpha(U_j)}\leq\\
\leq & |\Delta_j|_{C^\alpha(U_j)}\|1\otimes f-f\otimes 1\|_{C^0(U_j)}+\|\Delta_j\|_{C^0(U_j)}|1\otimes f-f\otimes 1)|_{C^\alpha(U_j)}\lesssim \\
\lesssim&\|f\|_{C^\beta(X)} j^{-\gamma},
\end{align*}
for 
$$\gamma=\min\left(1-\frac{\alpha}{\beta},\beta-\alpha\right).$$ 
We have that $\gamma>0$ if $\beta>\alpha$. We conclude that 
$$\|\Delta_j(1\otimes f-f\otimes 1)\|_{C^\alpha(X\times X)}\lesssim \|f\|_{C^\beta(X)}j^{-\gamma},$$
and the theorem follows.
\end{proof}

\begin{remark}
\label{contindiandinad}
If the linear map $T_j$ from the proof of Theorem \ref{approxomtienri} would be i) continuous $C_1(C^\beta(X))\to C_{2}(C^\alpha(X))$ in the projective norm and ii) $\Delta_j(1\otimes f-f\otimes 1)\to 0$ in $C^\alpha(X)\tilde{\otimes} C^\alpha(X)$ uniformly in $f\in C^\beta(X)$ for $\beta>2\alpha$, then the map $HH_1^{\rm red}(C^\beta(X))\to HH_1^{\rm red}(C^\alpha(X))$ is the zero map as soon as $\beta>2\alpha$. The argument dualizes to show the vanishing also of $HH^1_{\rm red}(C^\alpha(X))\to HH^1_{\rm red}(C^\beta(X))$. In particular, we could conclude that $HC_1^{\rm red}(C^\beta(X))\to HC_1^{\rm red}(C^\alpha(X))$ is the zero map as soon as $\beta>2\alpha$. The next result provides an example when the map $HC_1^{\rm red}(C^\beta(X))\to HC_1^{\rm red}(C^\alpha(X))$ is non-zero in the range $2\alpha>\beta>\alpha$.
\end{remark}

\begin{prop}
\label{lkandapdn}
Let $\alpha\in (1/2,1]$. Then it holds that $[z\otimes z^{-1}]\neq 0$ in $HC_1^{\rm red}(C^\alpha(S^1))$. In fact, there is a cyclic $1$-cocycle $c_1$ continuous on $C^\alpha(S^1)$ for $\alpha>1/2$ with 
$$[c_1].[z\otimes z^{-1}]\neq 0.$$
\end{prop}

\begin{proof}
Let $F$ denote the phase of the Dirac operator on $S^1$. The cyclic $1$-cocycle 
$$c(a,b):=\frac{1}{2\pi}\int_{S^1} a\mathrm{d}b=\mathrm{Tr}(F[F,a][F,b]),$$
extends by continuity to $C^\alpha(S^1)$ for $\alpha>1/2$. For details, see \cite[Chapter III.$2.\alpha$]{connebook} or \cite{goffodd}. A direct computation shows that 
$$[c_1].[z\otimes z^{-1}]=-1.$$
\end{proof}

The formula $c(a,b)=\mathrm{Tr}(F[F,a][F,b])$ shows that $c_1$ is continuous in the minimal tensor product for $C^\alpha(S^1)$ for $\alpha>1/2$. In particular, the inclusion $Z_1(C^\beta(S^1))\subseteq \overline{B_1(C^\alpha(S^1))}^{C^\alpha(S^1)\otimes^{\rm min}C^\alpha(S^1)}$ fails for $\alpha<\beta<2\alpha$ when $\alpha>1/2$. In light of Remark \ref{contindiandinad} we make the following conjecture. We find the problem interesting seeing that Proposition \ref{lkandapdn} would then demonstrate a cutoff behaviour.

\begin{conj}
Let $\alpha\in (0,1/2)$. Then it holds that $[z\otimes z^{-1}]= 0$ in $HC_1^{\rm red}(C^\alpha(S^1))$.
\end{conj} 

\section{Singular Chern characters of weakly summable Fredholm modules}

We now introduce a novel class of cocycles in cyclic and Hochschild cohomology using singular traces. They are associated with weakly summable Fredholm modules. Examples thereof will be computed below in Section \ref{examoedonaodhol}.

\begin{deef}
Let $(\pi,\mathcal{H},F)$ be a strict $(p+1,\infty)$-summable Fredholm module of parity $\#p$, for a $p\in \N$, on a Banach algebra $\mathfrak{A}$. Fix an extended limit $\omega$ and let $\mathrm{Tr}_\omega$ denote the associated Dixmier trace on $\mathcal{L}^{1,\infty}$. Define the $p+1$-cochain 
$$\mathfrak{h}_\omega(a_0,a_1,\ldots, a_{p+1})=p\mathrm{STr}_\omega\left(Fa_0[F,a_1]\cdots [F,a_{p+1}]\right),$$ 
and the $p$-cochain
$$\mathfrak{c}_\omega(a_0,a_1,\ldots, a_{p})=\mathrm{STr}_\omega\left(F[F,a_0]\cdots [F,a_{p}]\right).$$ 
\end{deef}

\begin{prop}
The $p+1$-cochain $\mathfrak{h}_\omega$ is a Hochschild $p+1$-cocycle and the $p$-cochain $\mathfrak{c}_\omega$ is a cyclic $p$-cocycle.
\end{prop}

\begin{proof}
One computes that 
\begin{equation}
\label{caomoamdao}
b\mathfrak{h}_\omega(a_0,a_1,\ldots, a_{p+1},a_{p+2})=p\mathrm{STr}_\omega\left([F,a_{p+2}]a_0[F,a_1]\cdots [F,a_{p+1}]\right)=0.
\end{equation}
The last equality follows from the fact that $[F,\mathfrak{A}]\subseteq \mathcal{L}^{(p+1,\infty)}(\mathcal{H})$ and $\mathcal{L}^{(p+1,\infty)}(\mathcal{H})^{p+1}\subseteq \mathcal{L}^1(\mathcal{H})$ on which any singular trace vanishes. That $\mathfrak{c}_\omega$ is a cyclic $p$-cocycle follows from standard considerations as in \cite[Chapter III]{connebook}.
\end{proof}

The fact that we are using singular traces manifests itself through the identity \eqref{caomoamdao} ensuring that $\mathfrak{h}_\omega$ is a Hochschild $p+1$-cocycle. This fact leads us to non-classical features such as the following theorem.

\begin{thm}
\label{somsdonadona}
Let $(\pi,\mathcal{H},F)$ be a $(p+1,\infty)$-summable Fredholm module of parity $\#p$, for a $p\in \N$, on a Banach algebra $\mathfrak{A}$. The classes $[\mathfrak{h}_\omega]\in HH^{p+1}(\mathfrak{A})$ and $[\mathfrak{c}_\omega]\in HC_\lambda^{p}(\mathfrak{A})$ satisfy
\begin{align*}
B[\mathfrak{h}_\omega]&=p[\mathfrak{c}_\omega]\in HC_\lambda^{p}(\mathfrak{A})\\
S[\mathfrak{c}_\omega]&=0\in HC_\lambda^{p+2}(\mathfrak{A}).
\end{align*}
\end{thm}

\begin{proof}
By definition, 
$$B\mathfrak{h}_\omega(a_0,a_1,\ldots, a_{p})=\mathfrak{h}_\omega(1,a_0,a_1,\ldots, a_{p})=p\mathrm{STr}_\omega\left(F[F,a_0]\cdots [F,a_{p}]\right)=p\mathfrak{c}_\omega(a_0,a_1,\ldots, a_{p}).$$
Since $[\mathfrak{c}_\omega]\in \im(B)$, $S[\mathfrak{c}_\omega]=0$ follows.
\end{proof}

Write $\mathcal{L}^{q,\infty}_0(\mathcal{H})$ for the closure of the finite rank operators in $\mathcal{L}^{q,\infty}(\mathcal{H})$. By definition, continuous singular traces vanish on $\mathcal{L}^{1,\infty}_0(\mathcal{H})$.

\begin{prop}
\label{blknaolknadadknlkadn}
Let $(\pi,\mathcal{H},F)$ be a strict $(p+1,\infty)$-summable Fredholm module of parity $\#p$, for a $p\in \N$, on a Banach algebra $\mathfrak{A}$. Assume that $\mathfrak{A}_0\subseteq \mathfrak{A}$ is a Banach subalgebra such that 
$$[F,a]\in \mathcal{L}^{p+1,\infty}_0(\mathcal{H}), \quad\forall a\in \mathfrak{A}_0.$$
Let $j:\mathfrak{A}_0\hookrightarrow \mathfrak{A}$ denote the inclusion. Then the classes $[\mathfrak{h}_\omega]\in HH^{p+1}(\mathfrak{A})$ and $[\mathfrak{c}_\omega]\in HC_\lambda^{p}(\mathfrak{A})$ satisfy
\begin{align*}
j^*[\mathfrak{h}_\omega]&=0\in HH^{p+1}(\mathfrak{A}_0)\\
j^*[\mathfrak{c}_\omega]&=0\in HC_\lambda^{p}(\mathfrak{A}_0).
\end{align*}
\end{prop}

\begin{proof}
The proposition follows from the definition of $\mathfrak{h}_\omega$ and the fact that Hölder estimates give us the inclusion $\mathcal{L}^{q,\infty}_0(\mathcal{H})\mathcal{L}^{q',\infty}(\mathcal{H})\subseteq \mathcal{L}^{1,\infty}_0(\mathcal{H})$ for any two conjugate exponents $q$ and $q'$.
\end{proof}

From Theorem \ref{ckadoandona} and \ref{somsdonadona} we can deduce the following.

\begin{thm}
\label{apakdpadj}
Let $(\pi,\mathcal{H},F)$ be an $(n,\infty)$-summable Fredholm module of parity $\#n-1$, for an $n\in \N$, on a Banach algebra $\mathfrak{A}$. Then we have a commuting diagram 
\[
\begin{tikzcd}
K_{n}^{\rm rel}(\mathfrak{A})\arrow[rr]\arrow[dd, "\ch^{\rm rel}_n"]&&K_n^{\rm alg}(\mathfrak{A})\arrow[dd, "\frac{1}{n!}D_n"] \\
&&\\
HC_{n-1}^\lambda(\mathfrak{A})\arrow[rr,"\frac{1}{n}B"]\arrow[dr,"\mathfrak{c}_\omega"] & & HH_{n}(\mathfrak{A})\arrow[dl,"\mathfrak{h}_\omega"]\\
&\C.&
\end{tikzcd}
\]

\end{thm}

The reader should compare Theorem \ref{apakdpadj} to the pairing of Connes-Karoubi's multiplicative character with the Connes-Karoubi sequence first studied in \cite{conneskaroubi}, and further developed by Kaad \cite{kaadcomparison,kaadcomp}. Connes-Karoubi's multiplicative character only maps $HH_{n}(\mathfrak{A})\to \C/\Z$. The quotient by $\Z$ is needed to make the map well defined since the Connes-Chern character of a Fredholm module does not lift to a Hochschild cocycle, a situation quite different from that of Theorem \ref{somsdonadona}.

\section{Examples of nontrivial singular Chern characters on Hölder algebras}
\label{examoedonaodhol}

In this section we will compute the singular Chern characters in examples and show that it produces non-trivial classes in cyclic cohomology. A result from \cite{gimpgoff2} allows us to produce an abundance of examples of weakly summable Fredholm modules on Hölder algebras. We write $\Psi^0(X;E)\subseteq \Bo(L^2(X;E))$ for the $*$-algebra of order zero classical pseudodifferential operators on a vector bundle $E\to X$, for more details on pseudodifferential operators see \cite{horIII,shubinbook}. 

\begin{thm}
\label{adpnapkdnasp}
Let $X$ be a compact $n$-dimensional Riemannian manifold, $P\in \Psi^0(X;E)$ a classical order $0$ pseudodifferential operator with $P=P^2=P^*$ and take $\alpha\in (0,1]$. Let $\pi$ denote the pointwise action of $C^{0,\alpha}(X)$ on $L^2(X;E)$. Then $(\pi,L^2(X;E),2P-1)$ is a strict $(n/\alpha,\infty)$-summable Fredholm module for $C^{0,\alpha}(X)$. 
\end{thm}

The proof of this result follows immediately from \cite[Theorem 2.1]{gimpgoff2}, but as discussed in \cite{gimpgoff2} this result was likely known before to experts. To compute with the associated singular cocycles, we will make use of the following result that follows from \cite[Section 2.2]{gimpgoff2} (see especially Theorem 2.15, Lemma 2.21 and Proposition 2.26 of \cite{gimpgoff2}).

\begin{thm}
\label{adpomopknad}
Let $X$ be a compact $n$-dimensional Riemannian manifold, $P\in \Psi^0(X;E)$ a classical order $0$ pseudodifferential operator with $P=P^2=P^*$ and $p\geq n-1$ an odd integer. Let $(e_k)_{k=1}^\infty\subseteq L^2(M;E)$ be an ON-basis consisting of eigenfunctions of some positive order elliptic pseudodifferential operator. 

The associated classes $[\mathfrak{h}_\omega]\in HH^{p+1}(C^{0,\frac{n}{p+1}}(X))$ and $[\mathfrak{c}_\omega]\in HC^{p}_\lambda(C^{0,\frac{n}{p+1}}(X))$ can be computed from the identity 
$$\mathfrak{h}_\omega(a_0,\ldots a_{p+2})=p2^{p+1}\lim_{N\to \omega} \frac{\sum_{k=1}^N \langle (2P-1)e_k,a_0[P,a_1]\cdots [P,a_{p+1}]e_k\rangle_{L^2(M;E)}}{\log(2+N)}.$$
Moreover, if $p>n-1$, the two cocycles $\mathfrak{h}_\omega$ and $\mathfrak{c}_\omega$ vanish when restricted to the small Hölder algebra $h^{\frac{n}{p+1}}(X)$.
\end{thm}

In the case of Lipschitz functions on manifolds, $p=n-1$, the Hochschild $n$-cocycle $\mathfrak{h}_\omega$ is computed by rather classical means. For $p>n-1$, the case of Hölder algebras, we arrive at some non-classical expressions below.

\begin{thm}
Let $X$ be a compact $n$-dimensional Riemannian manifold and $P\in \Psi^0(X;E)$ a classical order $0$ pseudodifferential operator with $P=P^2=P^*$. Let $\pi$ denote the pointwise action of $\mathrm{Lip}(X)$ on $L^2(X;E)$. Then $(\pi,L^2(X;E),2P-1)$ is a strict $(n,\infty)$-summable Fredholm module. Moreover, the Hochschild $n$-cocycle $\mathfrak{h}_\omega$ is given by 
\begin{align*}
\mathfrak{h}_\omega(a_0,a_1\ldots, a_n)=\frac{2^{n+1}}{(2\pi)^n} \int_{S^*X} \mathrm{Tr}_E\left((2p-1)a_0\{p,a_1\}\cdots \{p,a_n\}\right) \rd x\rd \xi,
\end{align*}
where $p$ is the principal symbol of $P$ and $\{\cdot,\cdot\}$ denotes the Poisson bracket. 
\end{thm}

For a detailed proof of the theorem, see \cite[Theorem 2.29]{gimpgoff2}. The idea in the proof is to use the work of Rochberg-Semmes  \cite{rochsemmes} to estimate $\|[P,a]\|_{\mathcal{L}^{n,\infty}}\lesssim \|a\|_{W^{1,n}}$ and to use that $C^\infty(X)\subseteq \mathrm{Lip}(X)$ is dense in the $W^{1,n}$-topology. Therefore, $\mathfrak{h}_\omega$ is determined by its restriction to  $C^\infty(X)$. In this case, the computation reduces to the Connes' trace formula \cite{connestrace}.

\subsection{Non-triviality in $HC_\lambda^1(C^{1/2}(S^1))$}
\label{apdnapidna}

Let us consider the Fredholm module defined from the Dirac operator on the circle. In this case, $F$ denotes the phase of the Dirac operator. We have that $F=2P-1$, where $P$ is the Szegö projector. In the Fourier basis $e_n(z)=(2\pi)^{-1/2}z^n$ we have $Fe_n=\mathrm{sign}(n)e_n$. We use the convention that $\mathrm{sign}(0)=1$. We can also describe $P=(F+1)/2$ from the Cauchy integral operator
$$Pf(z)=\lim_{r\to 1-}\frac{1}{2\pi i}\int_{S^1}\frac{f(w)\rd w}{w-rz}.$$
The following construction will be used heavily to prove non-triviality of cyclic cohomology classes.

\begin{prop}
\label{asifnaoidnadi0n}
Let $\alpha\in (0,1)$. The map 
\begin{align*}
\mathfrak{W}_\alpha&:\ell^\infty(\N)\to C^\alpha(S^1),\\
[\mathfrak{W}_\alpha&(c)](z):=\sum_{k=0}^\infty c_k2^{-\alpha k}z^{2^k}.
\end{align*}
is continuous and an isomorphism onto its range. 
\end{prop}

The proposition follows from Littlewood-Paley theory, see \cite{pellerbook} for a discrete version thereof adapted to the circle. A direct proof of Proposition \ref{asifnaoidnadi0n} can be found in \cite[Theorem 4.9, Chapter II]{zygbook}. Recall also from \cite[Theorem 3.5]{gimpgoff2}, that 
\begin{equation}
\label{compmpmpnpanf}
\mathrm{Tr}_\omega(P\mathfrak{W}_{1/2}(c_1)(1-P)[\mathfrak{W}_{1/2}(c_2)]^*P)=\omega\circ \log_2(c_1c_2^*).
\end{equation}
This computation can be derived from Theorem \ref{adpomopknad} using the Fourier basis.

\begin{thm}
\label{dsomaodand}
Let $P$ denote the Szegö projector on $S^1$, $\pi$ denote the pointwise action of $C^{1/2}(S^1)$ on $L^2(S^1)$ and consider the $(2,\infty)$-summable bounded Fredholm module $(\pi, L^2(S^1),2P-1)$ for $C^{1/2}(S^1)$.
\begin{enumerate}
\item The associated cyclic cohomology class $[\mathfrak{c}_\omega]\in HC_\lambda^{1}(C^{1/2}(S^1))$ is non-trivial and vanishes on $h^{1/2}(S^1)$.
\item The associated Hochschild cohomology class $[\mathfrak{h}_\omega]\in HH^{2}(C^{1/2}(S^1))$ pairs non-trivially with $K_2^{\rm alg}(C^{1/2}(S^1))$.
\end{enumerate}
\end{thm}

\begin{proof}
By Theorem \ref{apakdpadj}, the theorem follows if we can prove that $\mathfrak{h}_\omega\circ D_2(x)\neq 0$ for an $x\in \mathrm{im}(K_2^{\rm rel}(C^{1/2}(S^1))\to K_2^{\rm alg}(C^{1/2}(S^1)))$. Following \cite{kaadcomparison}, we take $x\in K_2^{\rm alg}(C^{1/2}(S^1))$ to be the Steinberg symbol of $\mathrm{e}^{\mathfrak{W}_{1/2}(c_1)}$ and $\mathrm{e}^{\mathfrak{W}_{1/2}(c_2)^*}$  and compute from Equation \eqref{compmpmpnpanf} that $\mathfrak{h}_\omega\circ D_2(x)=\omega\circ \log_2(c_1c_2^*)$. Clearly, $x\in \ker(K_2^{\rm alg}(C^{1/2}(S^1))\to K_2^{\rm top}(C^{1/2}(S^1)))=\mathrm{im}(K_2^{\rm rel}(C^{1/2}(S^1))\to K_2^{\rm alg}(C^{1/2}(S^1)))$ and we can take $c_1,c_2\in \ell^\infty(\N)$ such that $\omega\circ \log_2(c_1c_2^*)\neq 0$.
\end{proof}

From the proof of Theorem \ref{dsomaodand} and \cite{goffeubuabda}, we can even conclude the following generalization of \cite[Proposition 3.15]{gimpgoff2}. 

\begin{cor}
The convex set of cyclic cohomology classes 
$$\{[\mathfrak{c}_\omega]: \mbox{where $\omega$ is an extended limit}\}\subseteq HC_\lambda^{1}(C^{1/2}(S^1)),$$
associated with the Szegö projector on $S^1$, spans a subspace containing a copy of $(\ell^\infty(\N)/c_0(\N))^*$. The copy of $(\ell^\infty(\N)/c_0(\N))^*$ is separated by its pairing with $K_2^{\rm rel}(C^{1/2}(S^1))$.
\end{cor}

\subsection{Non-triviality in $HC_\lambda^3(C^{1/4}(S^1))$}

Let us now consider the case of even lower regularity, that by Theorem \ref{adpnapkdnasp} results in a higher summability degree. We have the following theorem.

\begin{thm}
\label{fourtedo}
Let $P$ denote the Szegö projector on $S^1$, $\pi$ denote the pointwise action of $C^{1/2}(S^1)$ on $L^2(S^1)$ and consider the $(4,\infty)$-summable bounded Fredholm module $(\pi, L^2(S^1),2P-1)$ for $C^{1/4}(S^1)$.

The associated cyclic cohomology class $[\mathfrak{c}_\omega]\in HC_\lambda^{3}(C^{1/4}(S^1))$ and the associated Hochschild cohomology class $[\mathfrak{h}_\omega]\in HH^{4}(C^{1/2}(S^1))$ are non-trivial and vanish on $h^{1/4}(S^1)$. They pair non-trivially with $K_4^{\rm rel}(C^{1/4}(S^1))$ and $K_4^{\rm alg}(C^{1/4}(S^1))$, respectively, as in Theorem \ref{somsdonadona}.
\end{thm}

\begin{proof}
Let us start with some preliminary computations. We note that by Theorem \ref{somsdonadona}, it suffices to prove that $\mathfrak{c}_\omega(x)\neq 0$ for some cyclic $3$-cycle $x$ in the range of $\ch_4^{\rm rel}$. For $a_0,a_1,a_2,a_3\in C^{1/4}(S^1)$, we introduce the notation
$$a_0\wedge a_1\wedge a_2\wedge a_3:=\sum_{\sigma\in S_3} (-1)^\sigma a_0\otimes a_{\sigma(1)}\otimes a_{\sigma(2)}\otimes a_{\sigma(3)},$$
which is a cyclic $3$-cycle for $C^{1/4}(S^1)$ in the range of $\ch_4^{\rm rel}$ being the relative Chern character of a higher Loday symbol, see \cite[Section 5.3]{kaadcomp}. We compute using Theorem \ref{adpomopknad} that for $a_0,a_1,a_2,a_3\in C^{1/4}(S^1)$, such that $a_0, a_1^*,a_2,a_3^*$ admit holomorphic extensions to the unit disc, we have  
\begin{align*}
\frac{1}{2}\mathfrak{c}_\omega(a_0\wedge a_1\wedge a_2\wedge a_3)=&\mathrm{Tr}_\omega(Pa_0(1-P)a_1Pa_2(1-P)a_3P)-\\
&-\mathrm{Tr}_\omega(Pa_0(1-P)a_3Pa_2(1-P)a_1P)=\\
=&\lim_{N\to \omega} \frac{1}{\log(N+2)}\sum_{k=0}^N \sum_{m=k}^\infty k a_{0,k}a_{2,m}\left(a_{1,-m}a_{3,-k}-a_{1,-k}a_{3,-m}\right),
\end{align*}
where 
$a_{j,k}$ denotes the $k$:th Fourier coefficient of $a_j$. 

Consider $c_0,c_1,c_2,c_3\in \ell^\infty(\N)$ defined from $c_2(j)=c_3(j)=1$ and $c_0(j)=c_1(j)=(-1)^j$. The functions $a_0=\mathfrak{W}_{1/4}(c_0)$, $a_1=\mathfrak{W}_{1/4}(c_1)^*$, $a_2=\mathfrak{W}_{1/4}(c_2)$, and $a_3=\mathfrak{W}_{1/4}(c_3)^*$ satisfy that $a_0, a_1^*,a_2,a_3^*$ admit holomorphic extensions to the unit disc, and the discussion above applies to give
\begin{align*}
\frac{1}{2}\mathfrak{c}_\omega(\mathfrak{W}_{1/4}(c_0)\wedge&\mathfrak{W}_{1/4}(c_1)^*\wedge\mathfrak{W}_{1/4}(c_2)\wedge\mathfrak{W}_{1/4}(c_3)^*)=\\
=&\lim_{N\to \omega\circ \log_2} \frac{1}{\log(2)N}\sum_{l=0}^N \sum_{m=l}^\infty 2^{l/2-m/2} \left((-1)^{m+l}-1\right)=\\
=&-\lim_{N\to \omega\circ \log_2} \frac{1}{\log(2)N}\sum_{l=0}^N\left(\frac{2^{-1/2}}{1+2^{-1/2}}+\frac{2^{-1/2}}{1-2^{-1/2}}\right)=\\
=&-\frac{1}{\log(2)}\left(\frac{2^{-1/2}}{1+2^{-1/2}}+\frac{2^{-1/2}}{1-2^{-1/2}}\right)=-\frac{2\sqrt{2}}{\log(2)}\neq 0.
\end{align*}
\end{proof}

\subsection{Non-triviality in $HC_\lambda^2(C^{2/3}(S^1\times S^1))$}

Let us end with a higher-dimensional computation, on the two-torus $S^1\times S^1$. In even dimensions, we need an even analogue of Theorem \ref{adpomopknad} which is proven ad verbatim to the odd case. We introduce the notation $E_N\subseteq \Z^2$ for the set of multiindices $\mathbbm{k}\in E_N$ such that $|\mathbbm{k}|^2\leq \lambda_N$ -- the $N$:th eigenvalue of the Laplacian.

\begin{thm}
\label{adnaodnaond}
Let $F$ denote the phase of the flat Dirac operator on $S^1\times S^1$, let $\pi$ denote the pointwise action of $C^{2/3}(S^1\times S^1)$ on $L^2(S^1\times S^1,\C^2)$ and consider the $(3,\infty)$-summable even bounded Fredholm module $(\pi, L^2(S^1\times S^1,\C^2),F)$ for $C^{2/3}(S^1\times S^1)$.

The associated cyclic $2$-cocycle $\mathfrak{c}_\omega\in Z^{2}_\lambda(C^{2/3}(S^1\times S^1))$ is given by 
\begin{align*}
c_\omega(a_0, a_1, a_2)=&\mathrm{Tr}_\omega(U[U^*,a_0][U,a_1][U^*,a_2])-\mathrm{Tr}_\omega(U^*[U,a_0][U^*,a_1][U,a_2])=\\
=& \lim_{N\to \omega} \frac{4i}{\log(2+N)}\sum_{\mathbbm{k}\in E_N} \sum_{\mathbbm{m},\mathbbm{n}\in \Z^2} \rho(\mathbbm{k},\mathbbm{m},\mathbbm{n})a_{0,-\mathbbm{n}}a_{1,\mathbbm{m}}a_{2,\mathbbm{n}-\mathbbm{m}},
\end{align*}
where $(a_{j,\mathbbm{m}})_{\mathbbm{m}\in \Z^2}$ denotes the Fourier coefficients of $a_j\in C^{2/3}(S^1\times S^1)$, and for $\mathbbm{k},\mathbbm{m},\mathbbm{n}\in \Z^2$
\begin{align*}
\rho(\mathbbm{k},\mathbbm{m},\mathbbm{n})=\frac{\mathbbm{m}\times \mathbbm{n}+(\mathbbm{m}-\mathbbm{n})\times \mathbbm{k}}{|\mathbbm{n}+\mathbbm{k}||\mathbbm{m}+\mathbbm{k}|}+\frac{\mathbbm{n}\times \mathbbm{k}}{|\mathbbm{k}||\mathbbm{k}+\mathbbm{n}|}+\frac{\mathbbm{k}\times \mathbbm{m}}{|\mathbbm{k}||\mathbbm{k}+\mathbbm{m}|}.
\end{align*}
\end{thm}

The reader can note that $\rho$ defines a smooth function on $\R^6\setminus \{0\}$ homogeneous of degree $0$.

\begin{proof}
We use the Fourier basis $e_{\mathbbm{k}}(z)=(2\pi)^{-1}z^{\mathbbm{k}}$, $\mathbbm{k}\in \Z^2$, for $L^2(S^1\times S^1)$. The operator $F$ decomposes as 
$$F=\begin{pmatrix} 0& U\\ U^*& 0\end{pmatrix},$$
where $U$ is the partial isometry on $L^2(S^1\times S^1)$ that in the Fourier basis acts as $Ue_{(k_1,k_2)}=\frac{k_1+ik_2}{|k_1+ik_2|}e_{(k_1,k_2)}$. We can identify the multiindex $\mathbbm{k}=(k_1,k_2)\in \Z^2$ with the complex number $\mathbbm{k}=k_1+ik_2$, so $Ue_{\mathbbm{k}}=\frac{\mathbbm{k}}{|\mathbbm{k}|}e_{\mathbbm{k}}$. 

Using an even analogue of Theorem \ref{adpomopknad}, and a degree argument we have that 
\begin{align*}
c_\omega(a_0, a_1, a_2)=&\mathrm{Tr}_\omega(U[U^*,a_0][U,a_1][U^*,a_2])-\mathrm{Tr}_\omega(U^*[U,a_0][U^*,a_1][U,a_2])=\\
=& \lim_{N\to \omega} \frac{4i}{\log(2+N)}\sum_{\mathbbm{k}\in E_N} \sum_{\mathbbm{m},\mathbbm{n}} \rho(\mathbbm{k},\mathbbm{m},\mathbbm{n})a_{0,-\mathbbm{n}}a_{1,\mathbbm{m}}a_{2,\mathbbm{n}-\mathbbm{m}},
\end{align*}
where 
\begin{align*}
\rho(\mathbbm{k},\mathbbm{m},\mathbbm{n})=&\frac{1}{2i}\frac{\mathbbm{k}}{|\mathbbm{k}|}\left(\frac{\overline{\mathbbm{k}}}{|\mathbbm{k}|}-\frac{\overline{\mathbbm{k}+\mathbbm{n}}}{|\mathbbm{k}+\mathbbm{n}|}\right)\left(\frac{\mathbbm{k}+\mathbbm{n}}{|\mathbbm{k}+\mathbbm{n}|}-\frac{\mathbbm{k}+\mathbbm{m}}{|\mathbbm{k}+\mathbbm{m}|}\right)\left(\frac{\overline{\mathbbm{k}+\mathbbm{m}}}{|\mathbbm{k}+\mathbbm{m}|}-\frac{\overline{\mathbbm{k}}}{|\mathbbm{k}|}\right)=\\
=&\mathrm{Im}\left(\frac{\mathbbm{k}}{|\mathbbm{k}|}\frac{\overline{\mathbbm{k}+\mathbbm{n}}}{|\mathbbm{k}+\mathbbm{n}|}+\frac{\mathbbm{k}+\mathbbm{m}}{|\mathbbm{k}+\mathbbm{m}|}\frac{\overline{\mathbbm{k}}}{|\mathbbm{k}|}+\frac{\mathbbm{k}+\mathbbm{n}}{|\mathbbm{k}+\mathbbm{n}|}\frac{\overline{\mathbbm{k}+\mathbbm{m}}}{|\mathbbm{k}+\mathbbm{m}|}\right)=\\
=&\frac{\mathbbm{m}\times \mathbbm{n}+(\mathbbm{m}-\mathbbm{n})\times \mathbbm{k}}{|\mathbbm{n}+\mathbbm{k}||\mathbbm{m}+\mathbbm{k}|}+\frac{\mathbbm{n}\times \mathbbm{k}}{|\mathbbm{k}||\mathbbm{k}+\mathbbm{n}|}+\frac{\mathbbm{k}\times \mathbbm{m}}{|\mathbbm{k}||\mathbbm{k}+\mathbbm{m}|}.
\end{align*}
Here we use the notation $z\times w=\mathrm{Im}(\bar{z}w)=-w\times z$ and that for complex numbers $z,w,u$ of modulus $1$,
$$w(\bar{w}-\bar{z})(z-u)(\bar{u}-\bar{w})=2i\mathrm{Im}(w\bar{z}+u\bar{w}+z\bar{u})=2i(z\times w+w\times u+u\times z).$$
\end{proof}

\begin{conj}
The cyclic cohomology class $[\mathfrak{c}_\omega]\in HC^{2}_\lambda(C^{2/3}(S^1\times S^1))$ and the associated Hochschild cohomology class $[\mathfrak{h}_\omega]\in HH^{3}(C^{2/3}(S^1\times S^1))$, constructed from the Fredholm module of Theorem \ref{adnaodnaond}, are non-trivial.
\end{conj}

\end{document}